\documentclass[12pt]{amsproc}

\usepackage{amssymb}
\usepackage{amsmath}
\usepackage{euscript}

\numberwithin{equation}{section}
\newtheorem{lemma}{Lemma}[section]
\newtheorem{theorem}{Theorem}[section]

\begin{document}

\title[Parabolic Equation with Nonlocal Diffusion]%
{Parabolic Equation\\ with Nonlocal Diffusion\\
without a Smooth Inertial Manifold}
\author{Alexander V. Romanov}
\address{National Research University Higher School of Economics}
\email{av.romanov@hse.ru}

\subjclass[2010]{Primary 35B42; Secondary 35B40, 35K68}

\keywords{inertial manifold, semilinear parabolic equation,
nonlocal diffusion}

\maketitle

\begin{abstract}
We construct an example of a one-dimensional parabolic
integro-differential equation  with nonlocal diffusion which does
not have asymptotically finite-dimensional dynamics in the
corresponding state space. This example is more natural in the
class of evolutionary equations of parabolic type than those known
earlier.
\end{abstract}

\section*{Introduction}

The theory of inertial manifolds is the most extreme implementation
of the concept, going back to Hopf~\cite{1}, of finite-dimensional
large-time behavior of solutions of distributed evolution systems
with dissipation. The concept implies that the eventual dynamics of
a dissipative system with infinitely many degrees of freedom can in
a sense be controlled by finitely many parameters. It turns out
then that the main object of study is a class of semilinear
parabolic equations with Hilbert state space. Paradoxically, even
though the existence of an inertial manifold has so far been
established only for a narrow class of such problems, known
examples in which the nonexistence of such a manifold is guaranteed
have been difficult to construct and look rather artificial. In any
case, no examples of this kind have been known yet for real
problems of mathematical physics. The present paper is a step in
this direction. Namely, we present a family of integro-differential
equations of parabolic type with nonlocal diffusion on the circle
such that these equations do not have a smooth inertial manifold.

\section{Preliminaries}

We consider evolution equations of the form
\begin{equation}\label{(1.1)}
  \partial _{t} u=-Au+F(u)
\end{equation}
with linear part~$-A$ and nonlinear part~$F$ in a real separable
Hilbert space~$(X,\|\cdot\|)$. The general theory of such
equations can be found in the book~\cite{2}. A closed linear
operator~$A$ on~$X$ with dense domain~$D(A)$ is said to be
\textit{sectorial} if it generates an analytic semigroup
$\{e^{-At}\} _{t>0} $; in this case, the spectrum $\sigma (A)$
lies in some half-plane $\operatorname{Re}\lambda>\delta$. The
property of being sectorial is stable under bounded perturbations.
Assuming without loss of generality that $\delta>0$, we denote the
one-sided scale of Hilbert spaces corresponding to~$A$ by
$\{X^{\alpha}\}_{\alpha\ge0} $, where $X^{\alpha}=D(A^{\alpha})$
and $\|u\|_{\alpha}=\|A^{\alpha}u\|$ for $u\in X^{\alpha}$; then
$X^{0}=X$, $X^{1} =D(A)$, and $X^{\beta}\subset X^{\alpha}$ for
$\beta>\alpha $.

From now on we make the following main assumptions
about~\eqref{(1.1)}.

(H1) The linear operator~$A$ is sectorial, its resolvent is
compact, and the spectrum~$\sigma(A)$ lies in the half-plane
$\operatorname{Re} \lambda>\delta>0$.

(H2) For some $\theta\in[0,1)$, the nonlinear function $F$
maps~$X^{\theta}$ into~$X$ and satisfies the estimate
\begin{equation*}
  \|F(u)-F(v)\| \leq L(r)\|u-v\|_{\theta}
\end{equation*}
for $\|u\|_\theta\le r$ and $\|v\|_\theta\le r$.

(H3) Equation~\eqref{(1.1)} generates a continuous dissipative
semiflow $\{\Phi_t\}_{t\ge0}$ in the space~$X^\theta$.

The dissipativity of the solution semiflow means that
\begin{equation*}
\sup \lim_{t\to +\infty} \|\Phi_{t}u\|_\theta \le a
\end{equation*}
uniformly with respect to $u$ on bounded subsets of~$X^\theta$. We
denote the closed ball $(\|u\|_\theta\le r)$ in the state
space~$X^\theta$ by~$B_{r} $ and say that the ball~$B_{a}$ is
\textit{absorbing}. The number~$\theta$ will be called the
\textit{nonlinearity exponent} of Eq.~\eqref{(1.1)}. The embeddings
$X^{\beta}\subset X^{\alpha}$, $0\leq\alpha<\beta$, are dense and
compact; in particular, ${\|u\|_{\alpha}\leq
C(\alpha,\beta)\|u\|_{\beta}}$ for $u\in X^{\beta}$. Under
conditions (H1)--(H3), one can readily establish (using the
constructions in~\cite[Theorem~3.3.6]{2}) that the evolution
operators~$\Phi_{t}$ are compact for~$t>0$.

In real problems, the operator~$A$ often proves to be self-adjoint,
and the compactness of the resolvent of~$A$ is typical of the case
of parabolic partial differential equations in bounded
domains~$\Omega \subset {\mathbb R}^{m} $.

A set $U\subset X^{\theta } $ is said to be \textit{invariant}
 if $\Phi_tU=U$ for $t>0$. The \textit{global attractor}
$\EuScript A$ of a semiflow $\{\Phi_{t}\} _{t\ge 0} $ is
defined~\cite{3,4} as the union of all entire (existing for $t\in
(-\infty,+\infty )$) bounded trajectories of the
infinite-dimensional dynamical system~\eqref{(1.1)} in the state
space~$X^{\theta}$. The global attractor (called simply the
attractor in what follows) is a connected compact (by virtue of the
compactness of the evolution operators $\Phi_{t}$) invariant set
in~$X^{\theta}$ and uniformly attracts the balls in the
space~$X^{\theta}$ as $t\to +\infty $. In particular, $\EuScript A$
contains all possible limit modes (equilibrium points, cycles,
invariant tori, etc.)\ of the solution semiflow. By the
\textit{smoothing property} of the parabolic equation, one has
$\Phi_{t}X^{\theta}\subset X^{1}$ for $t>0$, and hence every
invariant set (in particular, the attractor) lies in~$X^{1}$.

Let us modify the function $F(u)$ (without losing Lipschitz or
$C^{k}$-regularity, $1\le k\le \infty$) outside the absorbing
ball~$B_{a}$ in such a way that the new function~$\tilde{F}(u)$ be
identically zero outside the ball~$B_{a+1}$. This ``truncation''
procedure, described in detail in~\cite{4}, permits one to proceed
to the equation $u_{t} =-Au+\tilde{F}(u)$ with globally Lipschitz
function $\tilde{F}(u)$ and dissipative phase semiflow
in~$X^{\theta}$; this equation inherits the eventual dynamics
\eqref{(1.1)}. We assume that all this has already been done,
return to the original notation $F(u)$, and use the condition
\begin{equation}\label{(1.2)}
 \|F(u)-F(v)\| \leq  L\|u-v\|_{\theta}
\end{equation}
on the nonlinear component of Eq.~\eqref{(1.1)} in what follows.
Note~\cite{2} that the phase semiflow $\{\Phi_{t}\}$ inherits the
smoothness of the function $F(u)$.

Conditions (H1) and (H2) ensure the solvability of
Eq.~\eqref{(1.1)} in $X^{\theta}$ locally in $t>0$. The
dissipativity of the corresponding  dynamical system is a
technical but important fact.

\begin{lemma}\label{1.1}
If assumptions~\textup{(H1)} and~\textup{(H2)} hold and
$\|F(u)\|\le M$ for $u\in X^{\theta}$, then Eq.~\eqref{(1.1)} is
dissipative in~$X^{\theta}$.
\end{lemma}

\begin{proof}
Let us rewrite \eqref{(1.1)} in the form of the Duhamel integral
equation
\begin{equation*}
 u(t)=e^{-At} u(0)+\int_{0}^{t}e^{-A(t-\tau )}  F(u(\tau ))d\tau .
\end{equation*}
Since $\operatorname{Re} \sigma (A)>\delta >0$, it follows from
the well-known estimates $\|e^{-At}u\|_\theta\le Ce^{-\delta t}
\|u\|_\theta$ and $\|e^{-At}u\|_\theta\le Ct^{-\theta} e^{-\delta
t}\|u\|$ for arbitrary $u\in X^{\theta } $ that
\begin{equation*}
\| u(t)\| _{\theta }\le C e^{-\delta t}\| u(0)\| _{\theta } + C M
\int_{0}^{t}e^{-\delta (t-\tau )}  (t-\tau )^{-\theta } d\tau .
\end{equation*}
 We see that the norm
$\left\|u(t) \right\|_{\theta} $ remains bounded on the existence
interval of the solution~$u(t)$, whence it follows
\cite[Theorem~3.3.4]{2} that the solution can be extended
to~$[0,\infty)$. Thus, the original equation has the absorbing ball
$B_{a}\subset X^{\theta} $ of radius
\begin{equation*}
a = CM \int _{0}^{\infty }e^{-\delta s} s^{-\theta } ds,
\end{equation*}
and the proof of the lemma is complete.
\end{proof}

\section{Inertial Manifolds}

We will consider semilinear parabolic equations of the
form~\eqref{(1.1)} with self-adjoint linear operator~$A$, nonlinear
function $F\in C^{1}(X^{\theta},X)$, $0\leq \theta <1$, and
solution semiflow $\{ \Phi _{t} \} _{t\ge 0}$ in the state space
$X^{\theta}$. An \textit{inertial manifold} is a smooth or
Lipschitz finite-dimensional invariant surface $\EuScript M \subset
X^{\theta}$ containing the attractor $\EuScript A$ and
exponentially attracting all solutions $u(t)$ at large times.

Most of the known methods (starting from the fundamental
papers~\cite{5,6}) for constructing an $n$-dimensional inertial
manifold require the spectral jump condition
\begin{equation}\label{(2.1)}
 \lambda _{n+1} - \lambda _{n} > kL(\lambda _{n+1}^{\theta}
 + \lambda _{n}^{\theta } ),
\end{equation}
where $L$~is the constant in inequality~\eqref{(1.2)},
the~$\lambda _{n}$ are the eigenvalues of~$A$ arranged in
non-decreasing order (with regard for multiplicities), and $k$~is
some absolute constant.

It is well known \cite{7,8} that for ${\EuScript M}\in
\operatorname{Lip}$ one can take $k=1$, and this value is sharp
\cite{8}. The construction of a~$C^{1}$-smooth inertial manifold
usually assumes slightly larger values of~$k$, but there are
reasons to believe (see~\cite[p.~17]{9}) that $k=1$ is the optimal
constant in this case as well.

It was shown in~\cite{8} that the estimate~\eqref{(2.1)} with $k=1$
permits one to construct an $n$-dimensional Lipschitz inertial
manifold of Eq.~\eqref{(1.1)} in the form
\begin{equation*}
{\EuScript M} = \{ u \in  X^{\theta }\colon u  = y + h(y),\quad
y\in P_{n} X^{\theta }\},
\end{equation*}
where $P_{n}$ is the spectral projection of~$A$ corresponding to
the part $\{ \lambda _{1} \le \lambda _{2} \le \dotsm\le \lambda
_{n} \} $ of the spectrum, $h\colon P_{n} X^{\theta } \to (I-P_{n}
)X^{\theta} $ with $I=\operatorname{id}$, and moreover,
\begin{equation*}
 \|h(y)  -  h(y')\| _{\theta } \le d  \| y - y' \| _{\theta }
\end{equation*}
for $y,y'\in P_{n} X^{\theta } $. In this case, to each $u \in
X^{\theta} $ there corresponds a $\bar{u} \in {\EuScript M}$ such
that
\begin{equation*}
\| \Phi_{t}u - \Phi_{t}\bar{u} \| _{\theta } \le C \| u - \bar{u}\|
_{\theta} e^{-\gamma t}
\end{equation*}
for $t>0$ with $ \gamma =\lambda _{n+1} -\lambda _{n+1}^{\theta} K
>0$. The constant $C$ is independent of~$u$ and~$\bar{u} $. Since
the manifold ${\EuScript M}$ is invariant, it follows that
${\EuScript M}\subset X^{1} $.

The limit dynamics of the dynamical system $\{ \Phi _{t} \} _{t\ge
0} $ with state space $X^{\theta}$ is completely described by the
\textit{inertial form}
\begin{equation*}
   y_{t} = -Ay + P_{n} F(y + h(y)),\qquad y\in P_{n} X^{\theta },
\end{equation*}
which is an ordinary differential equation in $P_{n} X^{\theta }
\simeq {\mathbb R}^{n} $. In this case, one says that the original
equation~\eqref{(1.1)} is \textit{asymptotically $n$-dimensional}.

By using the spectral jump condition \eqref{(2.1)}, one can
establish the existence of an inertial manifold for a dissipative
equation~\eqref{(1.1)} with given linear part~$A$ and arbitrary
nonlinear function~$F$ satisfying the Lipschitz
condition~\eqref{(1.2)} under the ``spectrum sparseness''
assumption
\begin{equation}\label{(2.2)}
\sup_{n\ge 1} \frac{\lambda _{n+1} -\lambda _{n} }{\lambda
_{n+1}^{\theta }
 +\lambda _{n}^{\theta } } = \infty .
\end{equation}
If the linear part $-A$ of the parabolic equation~\eqref{(1.1)} is
the Laplace operator~$\Delta $ with standard boundary conditions in
$L^{2} (\Omega )$, $\Omega \subseteq {\mathbb R}^{m} $, then these
conditions become restrictive owing to the well-known asymptotics
$\lambda _{n} \sim c n^{2/m} $ of the eigenvalues $\lambda _{n} \in
\sigma (-\Delta )$. Attempts to sidestep condition~\eqref{(2.2)}
have only been successful in isolated special cases (e.g.,
see~\cite{10,11}). By now, the asymptotic finite-dimensionality has
not been established even for relatively simple problems such as
the parabolic equation
\begin{equation*}
 u_{t} = u_{xx}+f(x,u,u_{x})
\end{equation*}
on the circle or the reaction--diffusion equation
\begin{equation*}
 u_{t} = u_{xx}+f(x,u)
\end{equation*}
with standard boundary conditions in the disk.

On the other hand, extremely little is known about examples of
nonexistence of an inertial manifold for evolution equations
\eqref{(1.1)}. A system of two coupled one-dimensional parabolic
pseudodifferential equations without a smooth inertial manifold
was constructed in \cite{12} on the basis of the following
argument. Let $F'(u)$ be the Fr\'echet derivative of the smooth
mapping $F\colon X^{\theta} \to X$ at a point $u\in X^{\theta} $.
The linear operators $F'(u)$ are continuous from $X^{\theta}$ to
$X$ (i.e., $F'(u)\in \operatorname{End}(X^{\theta},X)$), and the
Lipschitz condition~\eqref{(1.2)} is equivalent to the estimate
$\|F'(u)\|_{\operatorname{op}}\leq L$ for $u\in X^{\theta}$. Let
$\sigma (T{(u}))$ be the spectrum of the unbounded linear operator
$T(u)=F'(u)-A$ in~$X$ with domain~$X^{1}$. Since $F'(u)=F'(u)A^{-
\theta}A^{\theta}$ with $\theta<1$, where $F'(u)A^{- \theta}\in
\operatorname{End}X$, it follows by~\cite[Sec.~1.4]{2} that the
operator $-T(u)$ inherits sectoriality, closedness, and the
compact resolvent property from the operator~$A$. Thus, $\sigma
(T(u))$ consists of eigenvalues of finite multiplicity. The number
$l(u)$ of positive eigenvalues in~$\sigma (T(u)) $ (counting
algebraic multiplicities) is finite. Finally, let $E$~be the set
of hyperbolic stationary points $u\in X^{\theta}$ of
Eq.~\eqref{(1.1)} for which the spectrum~$\sigma (T(u))$ does not
contain negative eigenvalues.

\begin{lemma}[\cite{12}]\label{2.1}
If the attractor~$\EuScript A$ of Eq.~\eqref{(1.1)} with nonlinear
function $F\in C^{1}(X^{\theta},X)$ is contained in some smooth
invariant finite-dimensional manifold ${\EuScript M}\subset
X^{\theta } $, then the number $l(u_{1} )- l(u_{2} )$ is even for
any $u_{1},u_{2}\in E$.
\end{lemma}

This fact was used in the recent papers~\cite{9,13} to obtain a
general construction of an abstract equation~\eqref{(1.1)} with
nonlinear function $F\in C^{\infty} $ and nonlinearity exponent
$\theta =0$ without a smooth inertial manifold. In the same
papers, a different (more delicate) argument was used to construct
an equation of the form~\eqref{(1.1)} with $F\in C^{\infty}$ not
admitting even a Lipschitz inertial manifold. The corresponding
results can apparently be extended to the general case of the
nonlinearity exponent $\theta \in [0,1)$. The counterexamples
in~\cite{9,12,13} are not sufficiently natural; it would be
desirable to present some physically meaningful semilinear
parabolic equation lacking asymptotic finite-dimensionality. To
some extent, this problem is solved in what follows. Note that the
corresponding example was announced by the author as early as
in~\cite{14}.

\section{Main Result}

 By $H^{\nu }$, $\nu \ge 0$, we denote the
generalized $L^{2}$ Sobolev spaces~\cite{15} of real functions on
the unit circle~$\Gamma $; in particular, $H^{0} =  H =L^{2}
(\Gamma)$. The differentiation operator $\partial _{x} u=u_{x} $ is
continuous from $H^{\nu +1}$ to $H^{\nu}$, and moreover, $\partial
_{x}\colon H^{1} \rightarrow H_{0}$, where $H_{0}\subset L^{2}
(\Gamma )$ is the subspace of functions with zero mean over~$\Gamma
$. For $\nu>1/2$, there are continuous embeddings $H^{\nu} \subset
C(\Gamma )$ and $H^{\nu +1} \subset C^{1}(\Gamma )$.

Consider the integro-differential parabolic equation
\begin{equation}\label{(3.1)}
 u_{t} =  ((I +  B) u_{x} )_{x} + f(x,u,u_{x} )
 +Ku,
\end{equation}
where $x \in\Gamma $. The bounded linear operators $K$, $
I=\operatorname{id}$, and $B=B^{*} $ act on the Hilbert space~$H$
with norm $\| \cdot \| $, and the function $f(x,s,p)$ defined on
$\Gamma \times {\mathbb R}^{2}$ is assumed to be infinitely smooth
but nonanalytic. Here the operator $I+ B$ plays the role of a
nonlocal diffusion coefficient, and the term $Ku$ can be
interpreted as a nonlocal source. More precisely, set
\begin{equation*}
 (Bh)(x) = \frac{1}{\pi } \int_{-\pi}^{\pi }
 \ln \left| \sin \frac{x+y}{2} \right| h(y)\, dy
\end{equation*}
and in addition
\begin{equation*}
(Jh)(x) = \frac{1}{2\pi }  \int_{-\pi}^{\pi } {\rm ctg}\,
\frac{x+y}{2} \, h(y)\, dy
\end{equation*}
for $h\in H$. The operator~$J$ is related to the Hilbert singular
integral operator
\begin{equation*}
({\EuScript G}h)(x) = \frac{1}{2\pi}\int _{-\pi}^{\pi} {\rm ctg}\,
\frac{y-x}{2} \, h(y)\, dy
\end{equation*}
by the formula $(Jh)(x)=({\EuScript G}h)(-x)$, $h\in H$. It is
well known~\cite[Chap.~6]{16}  that ${\EuScript G}1=0$ and
\begin{equation*}
 {\EuScript G}\colon \cos nx \rightarrow - \sin nx, \qquad
 {\EuScript G}\colon \sin nx \rightarrow  \cos nx
\end{equation*}
for integer $n\geq1$; hence $J1=0$ and
\begin{equation}\label{(3.2)}
 J\colon  \cos nx\to \sin nx, \qquad  J\colon \sin  nx\to\cos nx
\end{equation}
for such $n$.

The integral operators $J$ and~$B$ have the following properties:

(a)  $J\in \operatorname{End}H$ and $J^{2} =I$ on $H_{0}$;

(b) $B\in\operatorname{End}( H,H^{1})$ and $\partial _{x} B=J$ on
$H$.

\noindent Clearly, $J^{*} =J$, $B^{*} =B$, the operator~$B$ is
compact in~$H$, and
\begin{equation*}
 B\colon\cos
 nx\to -\frac{1}{n} \cos nx, \qquad   B\colon\sin nx\to
 \frac{1}{n} \sin nx
\end{equation*}
for~$n\geq1$. We see that the subspace $H_{0}$ is invariant
under~$B$, and the minimum eigenvalue of the restriction of~$B$
to~$H_{0}$ is~$-1$. Thus, the self-adjoint operator $I+B$ is
nonnegative on~$H_{0}$ and can be interpreted as a degenerate
nonlocal ``diffusion coefficient'' in the evolution
equation~\eqref{(3.1)}.

To represent Eq.~\eqref{(3.1)} in the standard form~\eqref{(1.1)},
we take $Au=u-u_{xx}$ with $D(A)=H^{2}$ and
\begin{equation}\label{(3.3)}
 F(u)=u+(Bu_{x})_{x}+f(x,u,u_{x})+Ku.
\end{equation}
Set $X=H$ and $X^{\alpha}=D(A^{\alpha})$ for $\alpha>0$. The
self-adjoint positive linear operator~$A$ on~$X$ has compact
resolvent, and $X^{\alpha} =H^{2\alpha}$. Note that
$\lambda_{n}\sim cn^{2}$ for the eigenvalues of~$A$, and hence the
spectrum sparseness condition~\eqref{(2.2)} does not hold even for
$\theta=1/2$, which is the minimum possible value of the
nonlinearity exponent in this situation.

\begin{theorem}\label{3.1}
For an appropriate choice of a function $f(x,s,p)\in C^{\infty}$
and a compact integral operator~$K$ with $C^{\infty}$-kernel,
Eq.~\eqref{(3.1)} generates a smooth dissipative semiflow
in~$X^{\theta }$, $\theta \in (3/4,1)$. The attractor of this
equation is not contained in any invariant finite-dimensional
$C^{1}$-manifold ${\EuScript M\subset X^{\theta}}$.
\end{theorem}

Take some bounded sequence of nonzero real numbers $\varepsilon
_{n} $, $n\ge 0$, and define the linear operator~$K$
in~\eqref{(3.1)} by the relations
\begin{equation}\label{(3.4)}
 K\colon\cos nx\to \varepsilon _{n} \sin (n+1)x,\qquad
 K\colon\sin (n+1)x\to -\varepsilon _{n} \cos nx.
\end{equation}

It is easily seen that $K\in \operatorname{End}X$ and $K^{*} =-K$.
If $\varepsilon _{n} \to 0$ exponentially as $n\to \infty $, then
$K$~is a compact operator and, by an argument like that
in~\cite[Sec.~6.7]{17}, an integral operator with
$C^{\infty}$-kernel in $X=L^{2} (\Gamma)$. Further, assume that
$|\varepsilon _{n}| \le \varepsilon _{0} <1 $ for $n\ge 1$; then
$\| K\| _{\operatorname{op}} =\varepsilon _{0}$, where $\| \cdot \|
_{\operatorname{op}}$ is the norm on the operator space
$\operatorname{End}X$.

Next, fix an arbitrary $\theta \in (3/4, 1)$. Since the embeddings
$X^{\theta} \subset C^{1} (\Gamma ) \subset C(\Gamma) \subset X$
are continuous, it follows that the mapping $u\to f(x,u,u_{x} )$
and hence the nonlinear component~$F(u)$ in~\eqref{(3.4)} belong to
the class $C^{\infty} (X^{\theta} ,X)$ for an arbitrary function
$f\in C^{\infty} (\Gamma \times {\mathbb R}^{2}) $. Moreover, the
function $F\colon X^{\theta} \rightarrow X$ satisfies the Lipschitz
condition on bounded sets in~$X^{\theta}$. Thus, the evolution
equation~\eqref{(3.1)} satisfies assumptions~(H1) and~(H2) with
nonlinearity exponent~$\theta$. The dissipativity of this equation
will be ensured by choosing a special structure of the
function~$f$.

Set
\begin{equation}\label{(3.5)}
 f(x,s,p)=\kappa \omega (s) w(p)+\varepsilon _{0}\gamma(s)
 +\varepsilon _{0} \eta (s) (1-\sin x)+\mu (s)
\end{equation}
with $\kappa \in {\mathbb R}$, $| \kappa|>1$. We assume that the
functions $\omega,\gamma,\eta,\mu\in C^{\infty}({\mathbb R})$
satisfy the conditions:
\begin{equation}\label{(3.6)}
\begin{gathered}
  \omega (z)=z,\; \gamma (z)=2z^{3}-3z^{2} ,\;
  \eta(z)=2z^{2}-z^{3},\;\mu
  (z)=0, \quad|z|\le 1 ;
\\
 \omega (z)=0, \quad \gamma(z)=0,\quad
\eta(z)=0, \quad \mu (z)=-z, \quad|z|\ge 2.
\end{gathered}
\end{equation}

Note that $\partial _{x} B \partial _{x} =J \partial _{x}$ on
$X^{1/2}$ and $J \partial _{x}\in\operatorname{End}(X^{1/2},X)$.
Let us momentarily represent the right-hand side of~\eqref{(3.1)}
in the form $F_{1} (u)-A_{1} u$, where
\begin{equation*}
 A_{1} =A-J\partial _{x}
-K, \qquad D(A_{1} )=D(A)=X^{1},
\end{equation*}
and $F_{1}(u) = u+f(x,u,u_{x} )$. Since $(A-J\partial _{x} )1=1$
and
\begin{align*}
A-J\partial _{x} &\colon \cos nx\to (1+n+n^{2} )\cos nx, \\
A-J\partial _{x} &\colon \sin nx\to (1-n+n^{2} )\sin nx
\end{align*}
for $n\geq1$, it follows that the minimum eigenvalue of the
self-adjoint operator $A-J\partial _{x} $ is~$1$. Thus, the
spectrum of the nonself-adjoint operator~$A_{1}$ lies in the
half-plane $\operatorname{Re} \lambda \geq 1-\varepsilon_{0}>0$,
and $(A-A_{1} )A^{-1/2} \in \operatorname{End}X$. The latter
property implies~\cite[Theorem~1.4.8]{2} that the operator~$A_{1} $
is sectorial in~$X$ and $D(A_{1}^{\alpha})=X^{\alpha}$ for all
$\alpha\geq 0$.

Since $\left| s+f(x,s,p) \right| \le \operatorname{const}$ on
$\Gamma \times {\mathbb R}^{2}$, it follows that $F_{1}\colon
X^{1/2} \rightarrow X$ and $\left\| F_{1} (u) \right\| \le
\operatorname{const}$ on~$X^{1/2}$. The adopted construction of
$f(x,s,p)$ ensures the estimates $\left| f_{s} \right| \le
\operatorname{const}$ and $\left| f_{p} \right| \le
\operatorname{const}$ on $\Gamma \times {\mathbb R}^{2}$ and hence
the global Lipschitz condition
\begin{equation*}
\| F_{1}(u)-F_{1}(v)\| \leq L \|u-v\|_{1/2}
\end{equation*}
for $u,v\in X^{1/2}$. Since $\theta>1/2$, it follows owing to the
continuity of the embedding $X^\theta \subset X^{1/2}$ that
$\left\| F_{1} (u) \right\| \le  \operatorname{const}$ on
$X^{\theta}$ and $F_{1}\in \operatorname{Lip}(X^\theta,X)$.
Lemma~\ref{1.1} guarantees the dissipativity of Eq.~\eqref{(3.1)}
in the state space~$X^{\theta}$, so that this equation satisfies
all assumptions~(H1)--(H3) with nonlinearity exponent~$\theta$.

Now let us return to the notation of Eq.~\eqref{(3.1)} with
nonlinear component~$F(x)$ of the form~\eqref{(3.3)} and with the
structure~\eqref{(3.5)}--\eqref{(3.6)} of the function~$f(x,s,p)$.
 The
linear operator $J\partial_{x}+K$ is continuous from~$X^{1/2}$
to~$X$ and so much the more from~$X^{\theta}$ to~$X$; consequently,
$F\in \operatorname{Lip}(X^{1/2},X)$ and $F\in
\operatorname{Lip}(X^\theta,X)$.

\begin{proof}[Proof \textnormal{of Theorem~\ref{3.1}}]
The linearization $T(u)=F'(u)-A $ of the vector field $F(u)-Au$ of
Eq.~\eqref{(3.1)} at a point $u\in X^{\theta}$ is (see Sec.~2) a
closed unbounded linear operator on~$X$ with compact resolvent and
with dense domain~$X^{1}$. This operator acts on functions~$h \in
X^{1}$ by the formula
\begin{equation}\label{(3.7)}
 T(u)h=h_{xx}+Jh_{x}+f_{s}(x,u,u_{x})h+f_{p}(x,u,u_{x})h_{x}+Kh.
\end{equation}
It follows from~\eqref{(3.6)} and~\eqref{(3.7)} that
\begin{equation}\label{(3.8)}
\begin{gathered}
f(x,0,0)=0, \quad f(x,1,0)=-\varepsilon _{0} \sin x, \quad f_{s}
(x,0,0)=f_{p} (x,0,0)=0,
\\
f_{s} (x,1,0)=\varepsilon _{0} (1-\sin x), \quad f_{p}
(x,1,0)=\kappa.
\end{gathered}
\end{equation}
Since $K1=\varepsilon_{0}\sin x$, it follows that $u_{0} =0$ and
$u_{1} =1$ are time-independent solutions of Eq.~\eqref{(3.1)}. We
have
\begin{equation*}
T(u_{0} )=Q+K, \qquad T(u_{1} )=Q+K+\kappa \partial _{x}
+\varepsilon _{0} (1-\sin x),
\end{equation*}
where $ Q =\partial _{xx} +J\partial _{x}$. We also see
from~\eqref{(3.2)} that
\begin{align*}
 Q&\colon \cos nx\to -(n^{2} +n) \cos nx, \quad n\ge 0,
\\
 Q&\colon \sin nx\to -(n^{2} -n) \sin nx, \quad n\ge 1.
\end{align*}

The two-dimensional subspaces
\begin{equation*}
 X_{n} = \{\cos nx, \sin (n+1)x \}, \qquad n\ge 0,
\end{equation*}
are invariant with respect to the operators~$Q$ and~$K$. According
to~\eqref{(3.4)}, the operator~$T(u_{0})$ can be represented in
each~$X_{n}$ by the matrix
\begin{equation*}
\begin{pmatrix} {-n^{2} -n} & {-\varepsilon _{n} } \\
{\varepsilon _{n} } & {-n^{2} -n} \end{pmatrix}.
\end{equation*}
Since $\varepsilon _{n} \ne 0$ and the subspaces $X_{n} $ form an
orthogonal basis in~$X$, it follows that the spectrum
$\sigma_{0}=\sigma(T(u_{0}))$ is purely nonreal.

Next, set $Q_{\kappa } =Q+\kappa \partial _{x}$ with the same
numerical parameter~$\kappa$ as in~\eqref{(3.5)}. By using the
results in~\cite{2} again, we find that the operator~$Q_{\kappa}$
is closed with domain~$X^{1}$ and has compact resolvent. In
addition, $Q_{\kappa}$ leaves invariant the subspaces
\begin{equation*}
Y_{n} = \{ \cos nx, \sin nx  \},\qquad  n\ge 1,
\end{equation*}
in~$X$ and is represented on each of these subspaces by the matrix
\begin{equation*}
\begin{pmatrix} {-n^{2} -n} & {\kappa n} \\ {-\kappa n} &
{-n^{2} +n} \end{pmatrix},
\end{equation*}
whose eigenvalues have the form $\lambda _{n} =-n^{2} \pm idn$ with
$d=(\kappa ^{2} -1)^{1/2} >0$. Since $Q_{\kappa } 1=0$, it follows
that the spectrum $\sigma(Q_{\kappa})$ consists of the point~$0$
and some complex numbers that are distant at least by~$d$ from the
real axis. The norm of the operator of multiplication by the
function $\xi (x)=1-\sin x$ in~$X=L^{2} (\Gamma )$ coincides with
the $\sup$-norm of this function, i.e., is equal to~$2$. Since
$\left\| K \right\| _{\operatorname{op}} =\varepsilon _{0} $, we
see that $\left\| K+\varepsilon _{0} \xi \right\|
_{\operatorname{op}} \le 3\varepsilon _{0} $. Now let us
treat~$T(u_{1} )$ as a bounded perturbation of the closed linear
operator $Q_{\kappa}$ with discrete spectrum. Using the stability
of root multiplicities \cite[Theorem~4.3]{18}, we can conclude
that, for sufficiently small $\varepsilon_{0}$, the spectrum
$\sigma_{1}=\sigma (T(u_{1}))$ contains at most one real eigenvalue
(counting algebraic multiplicities). It follows from
relations~\eqref{(3.7)} and~\eqref{(3.8)} that $T(u_{1}
)1=\varepsilon _{0} >0$; hence there indeed exists a (positive)
real eigenvalue.

Thus, for the given choice of the function $f(x,s,p)$ of the
form~\eqref{(3.5)}, \eqref{(3.6)} and the integral operator~$K$,
the spectra~$\sigma _{0}$ and~$\sigma_{1}$ of the
linearization~$T(u)$ of the vector field $F(u)-Au$ of
Eq.~\eqref{(3.1)} at the stationary points $u_{0}, u_{1}\in
X^{\theta}$ have the following properties: $\sigma _{0} \bigcap
{\mathbb R}=\phi $ and $\sigma_{1} \bigcap {\mathbb R}=\{
\varepsilon _{0} \} $, where $\varepsilon _{0}$ is the simple
positive eigenvalue of~$T(u_{1} )$. We see that the stationary
points~$u_{0}$ and~$u_{1}$ prove to be hyperbolic. If we denote the
number (counting multiplicities) of positive eigenvalues in the
spectrum of~$T(u)$, $u\in X^{\theta}$, by~$l(u)$, then $l(u_{0}
)=0$ and $l(u_{1} )=1$. Thus, the attractor of the semilinear
parabolic equation~\eqref{(3.1)} is not contained in any smooth
invariant finite-dimensional manifold ${\EuScript M}\subset
X^{\theta} $ by Lemma~\ref{2.1}, and the proof of Theorem~\ref{3.1}
is complete.
\end{proof}

The claim in Theorem~3.1 on the nonexistence of a smooth inertial
manifold for~\eqref{(3.1)} for an appropriate choice of the
nonlinear function $f(x,u,u_{x})$ remains valid if one replaces the
state space $X^{\theta}$, $\theta\in(3/4,1)$, with the more natural
space $X^{1/2}=H^{1}$ provided that one uses the weakened
version~\cite[p.~813]{10}  of the notion of differentiability of
nonlinear mappings (see also \cite[Definition~1.1]{12}).

\end{document}